\documentclass[12 pt,leqno,english]{article}
\title{Growth Results and Euclidean Ideals}
\author{Hester Graves}
\usepackage{euscript}
\usepackage{amssymb}
\usepackage{amsfonts}
\usepackage{amsthm}
\usepackage{amsmath}
\usepackage{enumerate}
\usepackage[english]{babel}
\usepackage{verbatim}

\newtheorem{theorem}{Theorem}
\newtheorem{defi}{Definition}
\newtheorem{prop}{Proposition}
\newtheorem{lemm}{Lemma}
\newtheorem{coro}{Corollary}

\newcommand{\Z}{{\mathbb{Z}}}
\newcommand{\Q}{{\mathbb{Q}}}
\newcommand{\N}{{\mathbb{N}}}

\newcommand{\Cl}{{\mathrm{Cl}}}

\newcommand{\p}{{\mathfrak{p}}}
\newcommand{\q}{{\mathfrak{q}}}

\newcommand{\I}{{\mathbb{I}}}

\newcommand{\Nm}{{\mathrm{Nm}}}

\newcommand{\A}{{\mathfrak{A}}}
\renewcommand{\O}{{\EuScript{O}}}

\begin{document}
\maketitle 
\begin{abstract}  Lenstra's concept of Euclidean ideals generalizes the Euclidean algorithm; a domain with a Euclidean ideal has cyclic class group, while a domain with a Euclidean algorithm has trivial class group.  This paper generalizes Harper's variation of Motzkin's lemma to Lenstra's concept of Euclidean ideals and then uses the large sieve to obtain growth results.  It concludes that if a certain set of primes is large enough, then the ring of integers of a number field with cyclic class group has a Euclidean ideal. 
\end{abstract} 

\section{Introduction}
If a Dedekind domain $R$ has a Euclidean algorithm, then $R$ is principal and the class group is trivial.  Lenstra generalized the concept of the Euclidean algorithm from a function defined on elements to a function defined on ideals via the concept of the Euclidean ideal.  If $R$ has a Euclidean ideal $C$, then $[C]$ generates the class group of $R$ and $\Cl_R$ is cyclic.  In the case where $R$ is the ring of integers of a number field, he proved something much stronger.

\begin{theorem}\cite{Lenstra}(Lenstra 1979) Suppose $K$ is a number field, $|\O_K^{\times}| = \infty$, and $C$ is an ideal of $\O_K$.  If one assumes the generalized Riemann hypothesis, then $C$ is a Euclidean ideal if and only if $\Cl_K =\left  \langle [C] \right \rangle.$
\end{theorem} 
In this paper, we will try to prove the above theorem in certain situations without using the Riemann hypothesis. We will do this by generalizing the machinery Harper used in his dissertation to study Euclidean rings. Our first result is the following.

\begin{theorem} \label{biquadratic} Suppose $K$ is a number field, $|\O_K^{\times}| = \infty$, and $C$ is an ideal of $\O_K.$  If $[C]$ generates the class group of $K$ and 
$$\{ \text{primes } \p \in \O_K | [\p] = [C], \O_K^{\times} \twoheadrightarrow (\O_K/ \p)^{\times} \} \gg \frac{x}{\log^2 x},$$
then $C$ is a Euclidean ideal.
\end{theorem}

\section{Euclidean Ideals}

\textbf{Notation:} Given a Dedekind domain $R$, we define $E:=\{\text{ideals } I: R \subset I\}.$  In other words, $E$ is the set of fractional ideals that contain $R$. 

Given a number field $K$, we denote its class group by $\Cl_K$, its class number by $h_K$, and its conductor by $f(K).$

\begin{defi}(\cite{Lenstra}) Suppose $R$ is a Dedekind domain.  If $C$ is an ideal of $R$, it is called $\textit{Euclidean}$ if there exists a function $\psi: E \longrightarrow W$, $W$ a well-ordered set, such that for all $I \in E$ and all $x \in IC \setminus C$, there exists some $y \in C$ such that $$\psi((x +y )^{-1} I C) < \psi(I).$$  We say $\psi$ is a $\textit{Euclidean algorithm for  } C$ and $C$ is a \textit{Euclidean ideal}.
\end{defi}

\begin{prop}(\cite{Lenstra}) If $R$ is a Dedekind domain with a Euclidean ideal $C$, then $[C]$ generates the class group of $R$.
\end{prop}

\begin{prop}(\cite{Lenstra}) If $K$ is an imaginary quadratic number field and the ring of integers $\O_K$ has a Euclidean ideal, then $K$ is either $\Q(i)$, $\Q(\sqrt{-2})$, $\Q(\sqrt{-3})$, $\Q(\sqrt{-5})$, $\Q(\sqrt{-7})$,  $\Q(\sqrt{-11})$, or $\Q(\sqrt{-15})$.
\end{prop}

\begin{defi}(\textbf{A Motzkin-type Construction for Ideals})

Given a Dedekind domain $R$ and some non-zero ideal $C$, we define 

$A_{0,C}:=  \{R\}, $

$A_{i,C}:= A_{i-1,C}\cup\left\{ I \left|\
  \begin{array}{c} I \in E \text{ and }
    \forall x\in IC\setminus C,\ \exists ~ y \in C
    \\ \mbox{such that}\ (x+y)^{-1}IC\in A_{i-1,C}
  \end{array}
  \! \!  \right.\right\}  \text{ for } i >0, $

$\text{and }A_C:= \cup_{i=0}^{\infty} A_{i,C}.$

\end{defi}  \index{$A_{i-1,C}$} \index{$A_C$} \index{ Motzkin's Construction for $C$} \index{$\phi_C$}
Note that the $A_{i,C}$'s are nested.

\begin{lemm}{{\textbf{(A Motzkin-type Lemma for ideals)}}} \label{modified motzkin's lemma I} 
\newline Suppose $R$ is a Dedekind domain and $C$ is a non-zero ideal. If the sets $A_C$ and $E$ are equal, then $C$ is a Euclidean ideal.  
\end{lemm}\index{Motzkin's Lemma for $C$}
\begin{proof}  We shall define $\phi_C: E \longrightarrow \N$ by $\phi_C(I) =i$ if $I \in A_{i,C} \setminus A_{i-1,C}.$ We will prove that $C$ is Euclidean by showing that $\phi_C$ is a Euclidean algorithm for $C$. Suppose that $I$ is an ideal in $E$ and that $x$ is an element in $IC \setminus C.$  Since the ideal $I$ is in $A_C$, there exists some $y$ in $C$ such that $(x+y)^{-1}IC$ is an element of $A_{\phi_C(I)-1, C}.$  We conclude that there exists some $y$ in $C$ such that 
$$\phi_C((x+y)^{-1}IC) \leq \phi_C(I)-1 < \phi_C(I).$$
The function $\phi_C$ is thus a Euclidean algorithm for $C$ and $C$ is a Euclidean ideal.
\end{proof}

\subsection{Properties of $IC/C$} 
\begin{defi}Suppose that $C$ is a non-zero ideal, that $I$ and $J$ are ideals in $E$, and that $\alpha$ is an element of $IC \setminus C.$  The ideal $J$ is \textit{similar to $\alpha$ modulo $IC$ and $C$}, or $J \sim \alpha \pmod{IC, C}$, if $J$ can be written as $ (\alpha + y)^{-1}IC $ for some $y $ in $C.$
\end{defi}

\begin{lemm}  Suppose that $R$ is a Dedekind domain, that $C$ is a non-zero ideal,  and that $I$ is an ideal in $E$.  If $x$ is an element of $IC$, then $\overline{x}$ generates $IC/C$ as an $R/I$ module if and only if $(x,C) = IC.$  Furthermore, if $(x,C)=IC$ and $\phi$ is any $R$-isomorphism from $(IC/C)^{-1}$ to $R/I^{-1}$, then $[\phi(x)]$ is a unit in $R/ I^{-1}.$ 
\end{lemm}
\begin{proof}  It is clear that the ideal $C$ is contained in the ideal $(x,C)$, which is itself contained in $IC.$  Let $\overline{x}$  generate $IC/C$ as an $R/I^{-1}$ module.  Therefore, for any $z$ in $IC \setminus C$, there exists some $a \in R$ such that $\overline{z}= [a]\overline{x}$. This implies that there exists some $y$ in $C$ such that $z=ax+y$ and so $IC$ is contained in $(x,C).$  We conclude that $(x,C)$ is equal to $IC.$

Let the ideal $(x,C) = IC$, so that $x $ is not in$C,$ and let $\overline{z}$ be any non-zero element of $IC/C$.  There exists some $a$ in $R$ and some $y$ in $C$ such that $z = ax+y$, implying that $\overline{z} = [a] \overline{x}.$  We conclude that $\overline{x}$ generates $IC/C$ as an $R/I$ module. 

Let the function $\phi:IC / C \longrightarrow R/I^{-1}$ be an isomorphism of $R/I^{-1}$ modules.  If $\overline{x}$ generates $IC/C$ as an $R/I^{-1}$ module, then $\phi(\overline{x})$ generates $R/I^{-1}$ as an $R/I^{-1}$ module.  In other words, $[\phi(x)]$ is a unit in $R/I^{-1}.$

\end{proof}
\begin{lemm} \label{rel prime}  Suppose that $R$ is a Dedekind domain, that $I$ is an ideal in $E$, and that $x$ is an element of  $IC$ such that $(x,C)= IC.$  If $\p$ is a prime ideal such that $v_{\p}(I^{-1}) \neq 0,$ then $\p$ and $xI^{-1}C^{-1}$ are relatively prime.
\end{lemm}
\begin{proof}If $\p$ is a prime ideal and  $I^{-1}$ is contained in $\p,$ then the ring $R$ is contained in $\p I$ and $C$ is contained in $\p I C.$  Suppose that the integral ideal $xI^{-1}C^{-1}$ is also contained in $\p,$ implying that  the element $x$ would be contained in $\p IC$, so the ideal $(x,C)$ would be contained in $\p IC.$  This is a contradiction, so the ideals $xI^{-1}C^{-1}$ and $\p$ must be relatively prime.
\end{proof}

\begin{lemm}\label{Dirichlet ideal class}Suppose that $K$ is a number field and that $I$ be an ideal in $E$, $I \neq R.$ If $x$ is an element of $IC$ and $(x,C)=IC$, then the set of  prime ideals $\p$ such that $\p^{-1} \sim x \pmod{IC,C}$ is of positive density in the set of all prime ideals.  In other words, there is a positive density of prime ideals $\p$ such that $\p^{-1} =(x+y)^{-1}IC$ for some $y$ in $C$.
\end{lemm}
\begin{proof}({Proof of Lemma \ref{Dirichlet ideal class}})  Suppose that $K^{I^{-1}}$ is the ray class field of $K$ of modulus $I^{-1}$ and that $x$ is an element of $IC$ such that $(x,C) = IC.$  We know by Lemma \ref{rel prime} that $(x,C)$ is relatively prime to any prime ideal $\p$ such that $\p$ divides $I^{-1},$  so the integral ideal $xI^{-1}C^{-1}$ is in the domain of the Artin map( i.e. $\I^{S(I^{-1})}$) for $K^{I^{-1}}$.  We shall call the support of the modulus $S(I^{-1})$ and we shall call this Artin map $\psi_{K^{I^{-1}}}.$

The Chebotarev density theorem implies that there is a positive density  of  prime ideals $\q$ in the preimage of $\psi_{K^{I^{-1}}}(xI^{-1}C^{-1})$  in the set of all prime ideals.  Each of these ideals $\q$ can be written as $(1+q)xI^{-1}C^{-1}$ for some $q$, where $v_{\p}(q) \geq v_{\p}(I^{-1})$ for all primes $\p$ in $S(I^{-1})$.  Note that $(x + xq) = (1+q)x = \q IC$, so that $x+xq$ is an element of $IC.$  We chose $x$ to be an element of $IC$, so $xq$ is also an element of $IC.$  As such, we know that $v_{\p}(xq) \geq v_{\p}(IC)$ for all primes $\p.$  Note this implies that $v_{\p}(xq) \geq v_{\p}(C)$ for $\p \notin S(I^{-1})$  because $v_{\p}(I)=0.$  Since $x$ is an element of $IC$, we know that $v_{\p}(x) \geq v_{\p}(IC)$ for all primes $\p.$  We know that $v_{\p}(q) \geq v_{\p}(I^{-1})$ for all primes $\p$ in $S(I^{-1})$, so that  $v_{\p}(xq) = v_{\p}(q)+v_{\p}(x) \geq v_{\p}(I^{-1}) + v_{\p}(IC) = v_{\p}(C)$ for primes $\p$ in $S(I^{-1}).$  We conclude that, because $v_{\p}(xq) \geq v_{\p}(C)$ for all primes $\p$, $xq$ is an element of $C$, so that $\q = (x +qx)I^{-1}C^{-1}$, where $qx$ is an element of $C$.  We conclude that   $\q^{-1} \sim x \pmod{IC,C}.$
\end{proof}
\begin{defi} For any $I \in E$  that is equivalent to $C^n$, $h_K \geq n >0$, fix some $x_I \in K^{\times}$ such that $I = x_{I^{-1}}^{-1} C^n.$
\end{defi} \index{$x_I$}
 Henceforth, assume that $C$ is an integral ideal.  This means that 
 \newline $(x_{I^{-1}}) = I^{-1}C^n$, so that $x_{I^{-1}}$ is an element of $C^n$ and is therefore an integer. 

\begin{lemm}\label{equiv}  Given two elements $x$ and $y $ in $C^n$, $n \geq 0$,  
\newline $x\equiv y (\textrm{mod} ~ \p C^n)$ if and only if $x \equiv y ~(\textrm{mod} ~ \p)$.  
\end{lemm}
\begin{proof} Since $x \equiv y ~(\textrm{mod} ~ \p C^n)$, then $x-y$ is in $\p C^n$ and is therefore in $\p$, so $x \equiv y ~(\textrm{mod} ~ \p)$.  

Conversely, suppose $x \equiv y ~(\textrm{mod} ~ \p)$.  Then $x -y \in \p$, but we also know that $x-y \in C^n$ because both $x, y \in C^n$.  As $x-y$ is in both $\p$ and $C^n$, it is in the intersection of $\p$ and $C^n$.  We know that $\p$ and $C$ are relatively prime, so $\p$ and $C^n$ are relatively prime, which means that their intersection is in fact their product.  We conclude that $x \equiv y ~(\textrm{mod} ~ \p C^m)$.  
\end{proof}

\begin{lemm} \label{in C}  Suppose that $\p$ is a prime ideal such that $[\p^{-1}] = [C^{n-1}].$ Given an element $y$ in $K^{\times}$, the product $yx_{\p}$ is in $\p C^n$ if and only if $y$ is an element of $C$.
\end{lemm}
\begin{proof}    If $y$ is in $C$, then $yx_{\p}$ is an element of $(\p C) (C^{n-1}) = \p C^n.$

  If $y x_{\p}$ is an element of $\p C^n$, then $(y x_{\p}) = I \p C^n$, for some non-zero integral ideal $I$.  Therefore $(y) = I x_{\p}^{-1}C^{n-1} \p C = I \p^{-1} \p C$ so that $(y) = I C$.  We conclude $y \in IC$, which implies that $y \in C.$
\end{proof} 

\begin{lemm}  \label{xpiso} Multiplication by $x_{\p}$ is an isomorphism from $\p^{-1} C / C $ to $ C^n / \p C^n.$
\end{lemm}  
\begin{proof}Given some $\beta$ in $\p^{-1} C$, we know that $x_{\p} \beta $ is in $\p C^{n-1} \p^{-1} C $, which is equal to $C^n$, so that multiplication by $x_{\p}$ is a map from $\p^{-1}C$ to $C^n$.  Let $b $ be an element of $C^n.$  Then $x_{\p}^{-1} b$ is an element of $\p^{-1} C^{1-n} C^n = \p^{-1}C$, so that multiplication by $x_{\p}$ is surjective.  

By Lemma \ref{in C},  the product $yx_{\p}$ is in $\p C^n$ if and only if $y $ is in $C$, so that $x_{\p}^{-1} \p C^n$ is $C$, and the kernel of the composition of multiplication by $x_{\p}$ and taking the quotient of $\p^{-1}C$ by $C$ is $\p C^n.$.  The isomorphism follows.  
\end{proof}

\section{Growth Results}
\begin{defi} Given a Dedekind domain $R$ and a non-zero ideal $C$, we define 

 $B_{0,C}:=\{R \},$

$B_{i,C} = B_{i-1,C}\cup\left\{ \p^{-1} \left|\
  \begin{array}{c} 
  \p \subseteq R\ \text{is prime, and} \\ 
    \forall x\in \p^{-1}C\setminus C,\ \exists y\in C \text{ such}\\
    \text{that}\ (x-y)^{-1}\p^{-1}C\in B_{i-1,C}
  \end{array}
  \right.\right\}$ for $i>0,$
  
and  $ B_C:= \bigcup_{i =0}^{\infty} B_{i,C}.$

\end{defi}

\begin{theorem}\label{joint functions}Suppose that $K$ is a number field and that $C$ is a non-zero ideal of $\O_K$.  If $B_C$ contains all ideals $\p^{-1}$ such that $\p$ is prime, then $C$ is a Euclidean ideal.
\end{theorem}

\begin{proof} The techniques of this proof follow those in \cite{Harper: 1}, however we use Lemma \ref{Dirichlet ideal class} instead of Dirichlet's theorem on primes in arithmetic progressions. 

Define $\omega$ to be the function mapping from $E$ to $\N$ such that $\omega(I)$ is the number of prime divisors of $I^{-1}$, with multiplicity.    If $\p$ is a prime ideal, define $\lambda(\p^{-1}):=i$ if $\p^{-1}$ is an element of $B_{i,C} \setminus B_{i-1,C}.$  Extend $\lambda$ to all of $E$ by additivity.   We then combine the two functions together as  $\phi: E \longrightarrow \N \times \N$, with $\phi(I) = ( \omega(I), \lambda(I))$, and order the image $\phi(E)$ lexicographically.  Because both $\omega$ and $\lambda$ are additive functions, the function $\phi$ is additive as well.   Note  that  $\phi(I) = (0,0)$ if and only if $I = \O_K$.  

In order to prove Theorem \ref{joint functions}, it suffices to show that $\phi$ is a Euclidean algorithm with respect to $C$, which we shall prove by induction.  Suppose that $\p$ is a prime ideal and that $\phi(\p^{-1}) =(1, i)$.  For any $x$ in $\p^{-1}C \setminus C$, there exists some $y$ such that  $(x+y)^{-1}\p^{-1}C$ is an element of $B_{i-1,C}$, so that $\phi((x+y)^{-1}\p^{-1}C)=(0,0)$ or $\phi((x+y)^{-1}\p^{-1}C) =(1, j)$, where $j <i.$

Let $I$ be an element of $E$ with $I^{-1}$ not a prime ideal.  Suppose that for all $J$ such that $\phi(J) < \phi(I)$ and for all $a$ in $JC \setminus C$, there exists some $y$ in $C$ such that $\phi((x+y)^{-1}JC) < \phi(J).$   If $x \in IC \setminus C$ and $(x,C) = IC$, we can apply Theorem \ref{Dirichlet ideal class}.  There exists some $y \in C$ such that $(x+y)^{-1}IC = \p^{-1}$, where $\p$ is a prime ideal.  Since $\phi(\p) $ is less than $\phi(I)$, the condition is satisfied.

If $x \in IC \setminus C$ and $(x,C)$ is not equal to $IC$, then we define $L$ to be $(I^{-1}, xI^{-1}C).$  The integral ideal $L$ contains $I^{-1}$, so  $I^{-1}L^{-1}$ has fewer prime divisors (counting multiplicities) than $I^{-1}$ and $\phi(IL)$ is less than $\phi(I)$.  There exists some $y$ in $C$ such that $\phi((x+y)^{-1}ILC) < \phi(IL).$  The additivity of $\phi$ implies that given two ideals $M$ and $N$, $\phi(MN) = \phi(M) + \phi(N)$, so  $\phi((x+y)IC) < \phi(I)$ and the condition holds.  We conclude that $\phi$ is a Euclidean algorithm for $[C]$ and that $[C]$ is a Euclidean ideal class.
\end{proof}

\begin{defi} Given the set $B_{i,C}$, we define 

$B_{i,C}(x):= \{ I \in B_{i,C}| \Nm(I^{-1}) \leq x \},$

$B_{i,C}^c:=\{ I \in E \setminus B_{i,C} |  [I] = [J] \text{ for any } J \in B_{i,C}\}, \text{ and}$

$B_{i,C}^c(x):=\{ I \in B_{i,C}^c| \Nm(I^{-1}) \leq x\}.$
\end{defi}

\begin{theorem}\label{main result}If $K$ is a number field such that $|\O_K^{\times}| = \infty$, $[C]$ generates $\Cl_K$, and   if 
$$| B_{1,C}(x) | \gg \frac{x}{\log^2 x},$$  then $E = A_C$ and $C$ is a Euclidean ideal.
\end{theorem}
\begin{coro}\label{main coro} If $K$ is an number field such that $|\O_K^{\times}| = \infty$,  if $[C]$ generates $\Cl_K$, and if
$$\left | \left \{
\begin{array}{c} \text{prime ideals } \\
\p \in \O_K 
\end{array}\left| \Nm(\p) \leq x,[\p] = [C], \O_K^{\times} \twoheadrightarrow (\O_K / \p)^{\times} \right. \right \} \right | \gg \frac{x}{\log^2 x},$$ then $C$ is a Euclidean ideal.
\end{coro}

\begin{proof} We will prove the corollary by proving that the ideal $\p^{-1}$ is in $B_{1,C}$ if and only if $[\p]=[C]$ and $\O_K^{\times} \twoheadrightarrow (\O_K/ \p)^{\times}.$
Clearly $\p^{-1}$ can only belong to $B_{1,C}$ if $[\p]=[C].$ For the rest of the proof, assume that $[\p] = [C].$ If $C^{h_K} = (c)$, then by Lemma \ref{xpiso}, multiplication by $\frac{x_{\p}}{c}$ is an isomorphism from $\p^{-1}C/C$ to $\O_K / \p.$   By definition, we know that $\p^{-1} \in B_{1,C}$ if and only if for all $\beta \in \p^{-1}C \setminus C$, there exists some $y \in C$ such that $(\beta + y)^{-1} \p^{-1} C = \O_K.$  We can rewrite  $(\beta + y)^{-1} \p^{-1} C = \O_K$ as $\frac{x_{\p}/c}{\frac{x_{\p}}{c} \beta + \frac{x_{\p}}{c} y} \p^{-1} C = \O_K,$
which is equivalent to the condition $\frac{x_{\p}}{c} \beta + \frac{x_{\p}}{c} y \in \O_K^{\times}$, where $\frac{x_{\p}}{c} \beta \in \O_K$ and $ \frac{x_{\p}}{c} y \in \p.$  We conclude that $(\beta + y)^{-1} \p^{-1} C = \O_K$ if and only if the image of $\beta$ under multiplication by $\frac{x_{\p}}{c}$ is congruent to a unit of $\O_K$ modulo $\p.$ 
\end{proof}
In order to prove  Theorem \ref{main result}, we will first need to state the Gupta-Murty bound and the Large Sieve for ideals.

\section{The Gupta-Murty Bound}
In order to state the Gupta-Murty bound, we need the following definitions.

 Given a prime ideal $\p$, define $q_{\p}: \O_K \longrightarrow \O_K/ \p$ to be the quotient map where $\p$ is the kernel.
  Using our new notation for the quotient map, we can now define an important constant for $\O_K$ and $\p$.
\begin{defi}  Given a prime ideal $\p$, we define $f(\p):= | q_{\p}(\O_K^{\times}) |.$ 
\end{defi} \index{$f(\p)$}
The constant $f(\p)$ is the size of the image of the units under the quotient group, which is  bounded above by the size of the set of  the units in $\O_K / \p$.  The constant $f(\p)$ is a special case of the following.

\begin{defi} Suppose that $K$ is a number field and $\p$ a prime ideal of $\O_K$.  If $\mathfrak{M}$ is a monoid in $\O_K$ such that $\mathfrak{M} \cap \p = \emptyset$, then  we define $f_{\mathfrak{M}}(\p)$ to be $|  q_{\p}(\mathfrak{M}) |.$
\end{defi}  \index{$f_{\mathfrak{M}}(\p)$}
The quantity $f(\p)$ is therefore the same as $f_{\O_K^{\times}}(\p).$ 

\begin{defi}  A set of elements $x_1, \ldots, x_n$ of $K$ is \textbf{multiplicatively independent}  if the only integer $n$-tuple $(a_1, \ldots, a_n)$ satisfying $x_1^{a_1} \cdots x_n^{a_n} = 1$ is $(0, \ldots, 0).$
\end{defi}

\begin{theorem}
{\textbf{(The Gupta-Murty Bound)}}
(\cite{Gupta})
\label{Gupta-Murty} 
If $\mathfrak{M}$ is a finitely-generated monoid in $\O_K$  containing $t$ multiplicatively independent elements, then 
$$| \{\p| f_{\mathfrak{M}}(\p) \leq y \}  |  << y^{\frac{t+1}{t}},$$
where the implied constant depends on $K$, $t$, and the generators of $\mathfrak{M}.$ 
\end{theorem} \index{Gupta-Murty bound}

\section{The Large Sieve for Euclidean Ideal Classes} \label{section The Large Sieve for Euclidean Ideal Classes}

The large sieve is at the heart of Harper's work on Euclidean rings.  In order to generalize his work and examine the asymptotic growth of the sets $B_{i,C}$, a generalized large sieve is needed.  Before the generalized version can be stated, however, we need the following definitions.

\begin{defi}  For each coset in the image of $\O_K^{\times}$ in $(\O_K/ \p)^{\times}$, choose one unit that maps to that coset.  Let $U(\p)$ be the collection of those units.  Note that $| U(\p) |$ is $f(\p)$.
\end{defi} \index{$f(\p)$} \index{$U(\p)$}

\begin{defi} Given $\A$, a finite set of non-associated integers, a prime ideal $\p$ and some $\alpha \in \O_K$, we define the following function
$$Z(\alpha, \p) := | \{x \in \A: x \equiv \alpha ~(\textrm{mod} ~ \p) \}|.$$
\end{defi} \index{$Z(\alpha, \p)$}

For our purposes, we want to apply the large sieve to sets of ideals, so that we look at how a finite set of ideals are distributed among the similarity classes of finite set of prime ideals, rather than how finite sets of elements are distributed among the equivalence classes of prime ideals.  We will therefore look at the function $Z(\alpha, \p,C)$ rather than $Z(\alpha, \p).$ 

\begin{defi} Suppose that $C$ is a non-zero integral ideal and that $n \in \Z^{+}$.
Let $A$ be a finite set of distinct fractional ideals $I$ in $E$, such that if $I$ and $J$ are in $A$, then $[I]=[J]$ .
If $\p$ is a prime ideal such that $[\p^{-1} ]=[IC^{-1}]$ for $I$ in $A$,  and if $\beta $ is in $\p^{-1}C $, we define 
$$Z(\beta, \p, C) := \left \{ \begin{array}{c}
|\{I \in A| \exists y \in C \text{ such that } (\beta +y)^{-1}\p^{-1}C = I\}| \\
\vspace{.1 in}
\text{ if } \beta \notin C \\
f(\p)|\{I \in A| \exists y \in C \text{ such that } (\beta +y)^{-1}\p^{-1}C = I\}|\\
 \text{ if } \beta \in C\\
\end{array}
\right .$$
\end{defi} \index{$Z(\beta, \p, C)$}

\begin{defi}  Given a finite set $A$, such that $A \subset \{I| I \in E \text{ and } [I]= [C^n]\}$, define $\A:= \{x_I| I \in A\}$.
\end{defi} \index{$\A$}

Using our notation above, we can prove the following theorem.
\begin{theorem} \label{first theo}  For $C$ a non-zero integral ideal, $\p$ a prime ideal that is relatively prime to $C$, and $\beta \in \p^{-1}C$, we have 
$$Z(\beta,\p,C) = \left \{ \begin{array}{c} 
\sum_{u \in U(\p)} Z(u\beta x_{\p}, \p) \text{ if } \beta \notin C\\
f(\p)Z(0, \p) \text{ if } \beta \in C\\
\end{array} 
\right .$$
\end{theorem}

\begin{proof}(Proof of Theorem \ref{first theo})  Using the elements defined above, we can rewrite the equation
$$I = (\beta + y)^{-1} \p^{-1}C$$ as 
$$x_I^{-1} C^n = (\beta + y )^{-1} x_{\p}^{-1}C^{n-1} C.$$

The above statement on ideals implies that 
$$ux_I = \beta x_{\p} + y x_{\p}  \text{ for some }u \in \O_K^{\times},$$
a statement on elements.  
Note that both $ux_I$ and $\beta x_{\p}$ are in $C^n$ and that $y x_{\p}$ is in $\p C^n$ if and only if $y$ is in $C$ by Lemma \ref{in C}.  Therefore, the statement that there exists some 
$$y \in C \text{ such that }I = (\beta + y)^{-1}\p^{-1} C$$ is equivalent to saying 
$$ux_I \equiv \beta x_{\p} ~(\textrm{mod} ~ \p C^n) \text{ for some }u \in \O_K^{\times}.$$
We know from Lemma \ref{in C} that this last condition is equivalent to 
$$x_I \equiv u' \beta x_{\p} ~(\textrm{mod} ~ \p) \text{ for some } u' \in \O_K^{\times}.$$  

Note that  $\beta$ is in $C$ if and only if 
$\beta x_{\p} \equiv 0 (\textrm{mod} ~ \p C^n)$, which implies that 
$ux_I \equiv 0 (\textrm{mod} ~ \p C^n)$, which is true if and only if 
$x_I \equiv 0 (\textrm{mod} ~ \p)$ by Lemma \ref{equiv}.   Since  there exists an element $y$ in $C$ such that $I = (\alpha + y)^{-1} \p^{-1}C$ if and only if 
$x_I \equiv 0 ~(\textrm{mod} ~ \p), $we conclude that if $\beta$ is in $C$, then 
$$Z(\beta, \p, C) = f(\p) | \{ I \in A: \exists y \in C \text{ such that } (\beta + y)^{-1}\p^{-1}C=I\}  |$$
$$= f(\p) | \{x_I \in \A: x_I \equiv 0 \, (\textrm{mod} ~ \p) \} | = f(\p) Z(\beta x_{\p}, \p).$$

If $\beta$ is not in $C$, this means there exists some $y$ in $C$ such that $I = (\beta + y)^{-1} \p^{-1}C$ if and only if $x_I \equiv u \beta x_{\p} ~(\textrm{mod} ~ \p)$ for some $u \in \O_K^{\times}.$  We conclude that if $\beta$ is not in $C$, then
$$Z(\beta, \p, C) = | \{I \in A: \exists y \in C \text{ such that } (\alpha +y)^{-1} \p^{-1}C = I \}|$$
$$= |\{x_I \in \A: x_I \equiv u \beta x_p  \text{ for some } u \in \O_K^{\times}\}|$$
$$= \sum_{u \in U(\p)} |\{x_I \in \A: x_I \equiv u \beta x_{\p} \, (\textrm{mod} ~ \p) \}|= \sum_{u \in U(\p)} Z(u\beta x_{\p}, \p).$$

\end{proof}

\begin{theorem} \label{sieve heart} If $C$ is a non-zero integral ideal; if $A$ is a finite collection of ideals, $A \subset E \cap [C^n]$; if $\p$ is a prime ideal;  and if $\p$ and $C$ are relatively prime,  then 
$$\sum_{\beta \in \p^{-1}C/C} \! \left ( \frac{Z(\beta, \p, C)}{f(\p)} - \frac{|A|}{\Nm( \p)}\right )^2 \leq \sum_{\alpha ~(\textrm{mod} ~ \p)} \! \! \left (Z(\alpha, \p) - \frac{|A|}{\Nm (\p)}\right )^2\hspace{-.1 in}.$$
\end{theorem}
\begin{proof}  From the above, if $\beta$ is in $C$, then 
$$ \left ( \frac{Z(\beta, \p, C)}{f(\p)} - \frac{|A|}{\Nm (\p)}\right )^2 
=\left (\frac{f(\p)Z(0, \p)}{f(\p)} - \frac{|A|}{\Nm (\p)}\right )^2,$$ which equals   

$$\left (Z(0,\p) - \frac{|A|}{\Nm (\p)}\right )^2\hspace{-.1 in}.$$

If $\beta$ is not in $C$, then   $$\left (\frac{Z(\beta, \p, C)}{f(\p)} - \frac{|A|}{\Nm (\p)} \right )^2 
= \left ( \frac{\sum_{u \in U(\p)} Z(u \beta x_{\p})}{f(\p)} - \frac{|A|}{\Nm( \p)} \right )^2\hspace{-.1 in}.$$
We can rewrite the right hand side as 
$$ \left (\sum_{u \in U(\p)} \hspace{-.1 in} \left (\frac{Z(u \beta x_{\p}, \p)}{f(\p)} - \frac{|A|}{f(\p) \Nm (\p) }\right ) \right )^2 \hspace{-.1 in},$$
which equals $$ \hspace{-.05 in} \frac{1}{(f(\p))^2}  \left (\sum_{u \in U(\p)} \hspace{-.1 in}\left (Z(u \beta x_{\p}, \p)-\frac{|A|}{\Nm( \p)} \right ) \right )^2.\hspace{-.1 in}$$

If we apply Cauchy-Schwartz, we see that the above is less than or equal to
$$ \hspace{-.25 in} \frac{1}{(f(\p))^2}   \hspace{-.05 in}\sum_{u \in U(\p)} \hspace{-.1 in} f(\p) \hspace{-.05 in} \left (Z(u \beta x_{\p}, \p)-\frac{|A|}{\Nm \p} \right )^2 \hspace{-.1 in}= \hspace{-.05 in} \frac{1}{f(\p)} \hspace{-.05 in} \sum_{u \in U(\p)} \hspace{-.1 in} \left (Z(u \beta x_{\p}, \p)-\frac{|A|}{\Nm \p} \right )^2\hspace{-.1 in}.$$ 

Summing up over all non-zero classes $\beta $ in $\p^{-1}C/C$ yields 

$$ \hspace{-.5 in} \sum_{\tiny{\begin{array}{c} \beta (\textrm{mod }  C)\\ \beta \not \equiv 0 (\textrm{mod } C) \end{array}}}
\hspace{-.3 in} \left ( \frac{Z(\beta, \p, C)}{f(\p)} - \frac{|A|}{\Nm \p} \right )^2 \hspace{-.1 in} 
\leq\frac{1}{f(\p)} \hspace{-.1 in}\sum_{\tiny{\begin{array}{c} \beta (\textrm{mod } C)\\ \beta \not \equiv 0 \end{array}}}  \hspace{-.1 in} \sum_{u \in U(\p)} \left (Z(u \beta x_{\p}, \p) - \frac{|A|}{\Nm \p} \right )^2$$

$$= \frac{1}{f(\p)} \sum_{u \in U(\p)} \hspace{-.1 in}\sum_{\tiny{\begin{array}{c} \beta (\textrm{mod }  C)\\ \beta \not \equiv 0 \end{array}}} 
\hspace{-.2 in} \left (Z(u \beta x_{\p}, \p) - \frac{|A|}{\Nm \p} \right )^2\hspace{-.1 in}.$$

The inner sum is independent of choice of $u$ so that the above is equal to 
$$\frac{1}{f(\p)} f(\p) \! \! \! \! \! \! \sum_{\tiny{\begin{array}{c} \beta (\textrm{mod } C)\\ \beta \not \equiv 0 \end{array}}} \! \! \! \!  \left (Z( \beta x_{\p}, \p) - \frac{|A|}{\Nm (\p)} \right )^2\hspace{-.1 in},$$
which can be further simplified to $$\sum_{\tiny{\begin{array}{c} \beta (\textrm{mod } C)\\ \beta \not \equiv 0 \end{array}}} \hspace{-.2 in} \left (Z(\beta x_{\p}, \p) - \frac{|A|}{\Nm (\p)} \right )^2\hspace{-.1 in} =
\hspace{-.3 in} \sum_{\tiny{\begin{array}{c} \alpha (\textrm{mod }\p),\\ \alpha \not \equiv 0 \end{array}}} \hspace{-.2 in} \left (Z(\alpha, \p) - \frac{|A|}{\Nm (\p)} \right )^2$$ by Lemma \ref{xpiso}.
Finally, by considering both cases at once, we get 
$$\sum_{\beta \in \p^{-1}C ~(\textrm{mod} ~ C)} \hspace{-.1 in}\left ( \frac{Z(\beta, \p, C)}{f(\p)} - \frac{|A|}{\Nm (\p)} \right )^2 \hspace{-.1 in}\leq \hspace{-.1 in} \sum_{\alpha (\textrm{mod } \p)} \hspace{-.1 in}\left (Z(\alpha, \p) - \frac{|A|}{\Nm (\p)} \right )^2\hspace{-.1 in}.$$
\end{proof}

\begin{theorem}{(\textbf{Large Sieve with Respect to $C$})}\label{Large Sieve with Respect to $C$}   \index{Large Sieve with Respect to $C$} 
\newline Suppose that $A$ and $P$ are  finite sets of fractional ideals, with $A \subset E \cap [C^n]$ and $P \subset \{\p: \p \text{ is prime, } [\p^{-1}]=[C^{n-1}]\}.$ If $X = \max_{I \in A} \Nm (I^{-1})$ and  $Q = \max_{\p^{-1} \in P} \Nm (\p)$, then
 
$$\sum_{\p \in P} \Nm (\p) \hspace{-.1 in}\sum_{\beta \in \p^{-1}C/C } \left (\frac{Z(\beta, \p, C)}{f(\p)} - \frac{|A|}{\Nm (\p)}\right )^2 \hspace{-.1 in} << (Q^2 + X) |A |.$$  The implied constant depends only on $K$, the ideal $C$, and on $n$.  
\end{theorem}

\begin{proof}  We know from Theorem \ref{sieve heart} that 
  $$\sum_{\beta \in \p^{-1}C /C} \hspace{-.1 in}\left ( \frac{Z(\beta, \p, C)}{f(\p)} - \frac{|A|}{\Nm (\p)} \right )^2  \hspace{-.05 in}\leq \hspace{-.15 in} \sum_{\alpha (\textrm{mod} ~ \p)} \! \!  \left (Z(\alpha, \p) - \frac{|A|}{\Nm (\p)}\right )^2.$$
This means that 
$$\sum_{\p \in P} \Nm (\p)  \hspace{-.1 in}\sum_{\beta \in \p^{-1}C /C} \hspace{-.1 in}\left ( \frac{Z(\beta, \p, C)}{f(\p)} - \frac{|A|}{\Nm (\p)}\right)^2 $$
$$\leq \sum_{\p \in P} \Nm (\p) \hspace{-.2 in} \sum_{\alpha (\textrm{mod} ~ \p)} \hspace{-.1 in}\left ( Z(\alpha, \p) - \frac{|A|}{\Nm (\p)} \right )^2\hspace{-.1 in}.$$
The maximum norm of any element in $\A$ is $$\text{max}_{x \in \A} \Nm (x) = \text{max}_{I \in A} \Nm (x_I) = \Nm (C^n) X.$$ 
Applying the large sieve, we know that 
$$\sum_{\p \in P} \Nm( \p) \hspace{-.1 in} \sum_{\alpha (\textrm{mod} ~ \p)} \hspace{-.1 in} \left (Z(\alpha, \p) - \frac{|A|}{\Nm (\p)}\right)^2 <<(Q^2 + \Nm(C)^n X)|\A |.$$
The sizes of $A$ and $\A$ are the same, so that  
 $$\sum_{\p \in P} \Nm (\p)\hspace{-.2 in} \sum_{\alpha (\textrm{mod} ~ \p)} \hspace{-.1 in} \left (Z(\alpha, \p) - \frac{|A|}{\Nm (\p)}\right )^2 << (Q^2 + \Nm(C)^nX)|A|,$$ where the implied constant only depends on the number field $K$.   
When we put this together with the earlier inequality, we see that 
$$\sum_{\p \in P} \Nm (\p) \hspace{-.1 in} \sum_{\beta \in \p^{-1}C /C} \hspace{-.1 in} \left ( \frac{Z(\beta, \p, C)}{f(\p)} - \frac{|A|}{\Nm (\p)}\right )^2  << (Q^2 + X)|A|,$$ with the implied constant now depending on both the choice of number field $K$, $C$, and $n$.

\end{proof}  

Harper did not use the large sieve in his paper, so much as a corollary of the large sieve.  In order to state our version of the corollary, we need the following definition.

\begin{defi}Let $\omega(\p):=~|  \{[\alpha] \in \p ^{-1}C/C | Z(\alpha, \p,C)=0\} |.$
\end{defi} \index{$\omega(\p)$}

\begin{coro}\label{newsievecoro} Let $A$ and $P$ be  finite sets of fractional ideals, with $A \subset E \cap [C^n]$ and $P \subset \{\p| \p \text{ is prime, } [\p^{-1}]=[C^{n-1}] \}.$ If $X = \max_{I \in A} \Nm (I^{-1})$ and  $Q = \max_{\p^{-1} \in P} \Nm (\p)$, then
$$\sum_{\p \in P} \frac{\omega(\p)}{\Nm(\p)} << \frac{Q^2 + X}{|A|},$$
where the implied constant depends only on $K$, $C$, and $n$.
\end{coro}
\begin{proof}  We know from Theorem \ref{Large Sieve with Respect to $C$} that 
$$\sum_{\beta \in \p^{-1}C/C} \left (Z(\beta,\p,C)-\frac{| A |}{\Nm(\p)} \right)^2 $$
$$\geq \sum_{ \left  \{ \tiny{ \begin{array}{c}\beta \in \p^{-1}C/C\\ Z(\alpha, \p,C) =0 \end{array}} \right \} } 
\left (Z(\alpha,\p,C)-\frac{|A|}{\Nm(\p)} \right)^2=\frac{|A |^2 \omega(\p)}{\Nm(\p)^2}.$$  
Since the quantity $\frac{|A|^2 \omega(\p)}{\Nm(\p)^2}$ is less or equal to than 
$$\sum_{\beta \in \p^{-1}C /C} \left (Z(\alpha,\p,C)-\frac{|A|}{\Nm(\p)} \right)^2 ,$$ we know that 

$$\sum_{\p \in P} \Nm (\p) \frac{|A|^2\omega(\p)}{\Nm(\p)^2} << (Q^2 + X)|A|$$
and therefore $$\sum_{\p \in P} \frac{|A|^2 \omega(\p)}{\Nm(\p)} <<(Q^2 +X)|A |.$$
We conclude that $$\sum_{\p \in P} \frac{\omega(\p)}{\Nm(\p)} << \frac{Q^2 + X}{|A|}.$$
\end{proof}

\section{Proof of Theorem \ref{main result}}
\begin{proof}(Proof of Theorem \ref{main result}) First we shall prove that if 

$|B_{1,C}(x)|  \gg \frac{x}{\log^2 x},$ then
$| B_{2,C}(x)| \sim \frac{x}{h_K \log x}.$  We shall define $A$ to be the set $B_{1,C}(x^2)$, define $X \leq x^2$, and   $P$ to be the set $B_{2,C} ^c (x)$, so  that $Q \leq x$.  By Corollary \ref{newsievecoro}, we know that
$$\sum _{\p \in P} \frac{\omega(\p)}{\Nm(\p)}<< \frac{Q^2 + X}{|A|} << \frac{2x^2}{\frac{x^2}{\log^2 (x^2)}} = \frac{2x^2}{\frac{x^2}{2 \log^2(x)}}<< \log^2(x).$$ 
Note that if $I$ is in $B_{1,C} ^c$, then there is some $x $ in $ IC \setminus C$ such that there is no $y $ in $ C$ such that $(x +y)^{-1}IC =\O_K^{\times}$.  Given any $u \in \O_K ^{\times}$, we know that $(ux + uy)^{-1}IC =  (x+y)^{-1}IC$ and that $uy$ is in $C$.  Therefore, if $Z(x, \p^{-1}, C)$ is zero, then so is $Z(ux, \p^{-1}, C)$ for any unit $u$.  
This means that if there is one $x$ in $\p^{-1}C \setminus C$ such that $Z(x, \p) = 0$, then there are at least $f(\p)$ cosets with elements $e$ such that $Z(e, \p) = 0$, and the function $\omega(\p)$ is divisible by $ f(\p)$.  

If $\O_K$ has infinitely many units, then 
$$|\{ \p | \Nm( \p) \leq x \text{ and }f(\p) \leq\ \Nm( \p)^{\frac{1}{2} - \epsilon } \}|<< x^{1 - 2 \epsilon}$$ by Theorem \ref{Gupta-Murty}
and therefore 
$$|\{ \p | \Nm( \p) \leq x \text{ and }f(\p) \leq \Nm( \p)^{\frac{1}{2} - \epsilon } \}| = o\left (\frac{x}{\log x}\right ).$$
The above estimate implies that
$$\sum_{\tiny{\begin{array}{c}
\p^{-1} \in P(x)\\
 f(\p) > \Nm(\p)^{\frac{1}{2} - \epsilon} \end{array}}}
    \frac{\omega(\p)}{\Nm(\p)} \geq 
 \sum_{\tiny{\begin{array}{c}\p^{-1} \in P^c(x)\\ f(\p) > \Nm(\p)^{\frac{1}{2} - \epsilon} \end{array}}}
\frac{f(\p)}{\Nm(\p)}$$
$$ >\frac{| \{ \p \in B_{2,C}^c: f(\p) > \Nm(\p)^{\frac{1}{2} - \epsilon} \} |}{x^{\frac{1}{2} + \epsilon}}.$$  
By combining this with the first bound in the proof, we know that
$$\log^2(x) \gg \frac{|\{ \p \in B_{2,C}^c: f(\p) > \Nm(\p)^{\frac{1}{2} - \epsilon} \}|}{x^{\frac{1}{2} + \epsilon}}.$$  
Multiplying both sides by $x^{\frac{1}{2} + \epsilon}$  yields 
$$| \{ \p \in B_{2,C}^c: f(\p) > \Nm(\p)^{\frac{1}{2} - \epsilon} \}| = o\left (\frac{x}{\log(x)}\right ).$$  
Since the size of $B_{2,C}^c(x)$ is equal to 
$$|\{\p \in B_{2,C}^c(x)| f(\p) > \Nm(\p) ^{\frac{1}{2} - \epsilon}\}| + 
|\{ \p \in B_{2,C}^c(x)| f(\p) \leq \Nm(\p)^{\frac{1}{2} - \epsilon} \}|,$$ 
we know the size of $B_{2,C}^c(x)$ is less than or equal to
$$|\{\p \in B_{2,c}^c(x)| f(\p) > \Nm(\p) ^{\frac{1}{2} - \epsilon}\}| + |\{ \p | \Nm(\p) \leq x, f(\p) \leq \Nm(\p)^{\frac{1}{2} - \epsilon} \}|$$ 
and can thus conclude $|B_{2,C}^c(x)|= o \left ( \frac{x}{\log(x)}\right)$
 and $$|B_{2,C} (x)| \sim \frac{x}{h_K \log x}.$$

Suppose that $\p$ is prime, that $[\p] = [C^{n+2}]$, and that $x $ is an element of $\p^{-1}C \setminus C.$ 
There is a positive density of prime ideals $\mathfrak{q}$ such that $\mathfrak{q}^{-1} = (x+y)^{-1} \p^{-1}C$ for some $y$ in $C$ by Theorem \ref{Dirichlet ideal class}.  The density of primes in $[C^{n+1}]$ in the set of all primes is $\frac{1}{h_K}$, which is the same as the density of primes $\mathfrak{q}$ such that $\mathfrak{q}^{-1}$ is in $B_{2, C}$, so there exists some prime $\mathfrak{q}$ in $B_{2,C}$ such that $\q^{-1} = (x+y)^{-1}\p^{-1}C$ for some $y$ in $C$.  We conclude that $\p^{-1}$ is in $B_{3, C}.$  
\end{proof}
\textbf{Acknowledgments} This paper grew out of the author's dissertation, which would not be completed were it not for the author's advisor, Nick Ramsey, who supported and encouraged her.  The author would like to thank E. Hunter Brooks, Chris Hall, Malcolm Harper, Jeffery Lagarias, and Chris Skinner for their help.

\end{document}